\documentclass[12pt]{amsart}
\usepackage{color,verbatim}

\newcommand{\Bbbk}{\mathbb{K}}
\newcommand{\xvect}{\mathbf{x}}

\newcommand{\Supp }{\textup{Supp }}
\newcommand{\Spec}{\textup{Spec }}
\newcommand{\Tor}{\textup{Tor }}

\newcommand{\Ad}{\textup{Ad}}
\newcommand{\dd}{d}
\newcommand{\Y}{\mathcal{Y}}

\newtheorem{theorem}{Theorem}
\newtheorem{proposition}[theorem]{Proposition}
\newtheorem{lemma}[theorem]{Lemma}
\newtheorem{example}[theorem]{Example}
\newtheorem{remark}[theorem]{Remark}

\title{A formula for the cohomology and $K$-class of a regular Hessenberg variety}

\author{Erik Insko, Julianna Tymoczko, and Alexander Woo}
\thanks{The second author was partially supported by NSF Grant DMS-1362855.  The third author was partially supported by Simons Collaboration Grant 359792.}

\begin{document}

\begin{abstract}
Hessenberg varieties are subvarieties of the flag variety parametrized by a linear operator $X$ and a nondecreasing function $h$.  The family of Hessenberg varieties for regular $X$ is particularly important: they are used in quantum cohomology, in combinatorial and geometric representation theory, in Schubert calculus and affine Schubert calculus.  We show that the classes of a regular Hessenberg variety in the cohomology and $K$-theory of the flag variety are given by making certain substitutions in the Schubert polynomial (respectively Grothendieck polynomial) for a permutation that depends only on $h$. Our formula and our methods are different from a recent result of Abe, Fujita, and Zeng that gives the class of a regular Hessenberg variety with more restrictions on $h$ than here.
\end{abstract}

\maketitle

Fix an algebraically closed field $\Bbbk$, let $G=GL_n(\Bbbk)$, and let $B$ be the subgroup of upper triangular matrices.  The {\em flag variety} $G/B$ can be thought of as the moduli space of {\em complete flags}, which are chains $\{0\}=F_0\subset F_1\subset\cdots\subset F_{n-1}\subset F_n=\Bbbk^n$ of subspaces of a fixed vector space $\Bbbk^n$ so that $\dim F_i=i$ for each $i$.  The {\em Hessenberg variety} $\Y_{X,h}\subseteq G/B$ is defined as follows.  Fix a linear operator $X$ and a function $h:\{1,\ldots,n\}\rightarrow \{1,\ldots,n\}$ such that $h(i)\geq i$ and $h(i)\leq h(i+1)$ for all $i$.  Then
$$\Y_{X,h}=\{F_\bullet\mid XF_i\subseteq F_{h(i)}\ \forall i\}.$$  Note that $\Y_{X,h}$ and $\Y_{X',h}$ are isomorphic if $X$ and $X'$ are similar operators.  
 
The Hessenberg varieties defined by regular operators form an important class of Hessenberg varieties.  A regular operator is an operator $X$ such that the generalized eigenspaces of $X$ have distinct eigenvalues.  De Mari, Procesi, and Shayman first defined and studied Hessenberg varieties in the case when $X$ is regular and semisimple ~\cite{DMS88,DMPS92}. Regular semisimple Hessenberg varieties govern an important geometric representation connected to the Stanley-Stembridge conjecture through a conjecture of Shareshian and Wachs~\cite{SW16} that was recently proven by Brosnan and Chow~\cite{BC18} and almost simultaneously by Guay-Paquet~\cite{GP15}.  
Regular nilpotent Hessenberg varieties are also extremely important: Kostant proved that with one particular $h$, they can be used to construct the quantum cohomology of the flag variety~\cite{Kos96} (see also Rietsch's work~\cite{Rie03}).  More general regular nilpotent Hessenberg varieties have since been studied prolifically, for example in~\cite{AHHM17, ADGH18, Dre15, MT13, ST06}. 
 
In this paper, for any $h$ and regular operator $X$, we give formulas for the class of $\Y_{X,h}$ in the cohomology $H^*(G/B)$ and $K$-theory $K^0(G/B)$ (or Grothendieck ring) of the flag variety.  Anderson and the second author gave a different formula for the cohomology class using degeneracy loci and degeneration arguments~\cite{AT10}.  The first two authors computed some of the coefficients in the Schubert expansions of these cohomology classes using intersection theory~\cite{IT16}.  Recently, Abe, Fujita, and Zeng gave a formula in K-theory~\cite[Cor. 4.2]{AFZ18}.  We use a different approach from all of the previous authors, following ideas of Knutson and Miller that use commutative algebra and the pullback of $G/B$ to $G$~\cite{KM05}; this gives us a different formula than that of Abe, Fujita, and Zeng.

Given $h$, we define a permutation $w_{h}\in S_{2n}$ as follows.  We let $w_h(i+h(i))=n+i$ and put $1,\ldots,n$ in the other entries in order.  
Let $x_j$ denote the class of the $j$-th tautological line bundle (in $K^0(G/B)$) or the Chern class of its dual (in $H^*(G/B)$).  Then our main result is:
\begin{theorem} \label{theorem: main theorem}
Suppose $X$ is a regular operator.  The class of the the Hessenberg variety $\Y_{X,h}$ in $K^*(G/B)$ is represented by
$$[\Y_{X,h}] = \mathfrak{G}_{w_h}(x_1,\ldots,x_{h(1)},x_1,x_{h(1)+1},\ldots,x_{h(2)},x_2,x_{h(2)+1},\ldots,x_{h(n)},x_n)$$
where $\mathfrak{G}$ is the Grothendieck polynomial.  Similarly, the class in $H^*(G/B)$ is represented by
$$[\Y_{X,h}] = \mathfrak{S}_{w_h}(x_1,\ldots,x_{h(1)},x_1,x_{h(1)+1},\ldots,x_{h(2)},x_2,x_{h(2)+1},\ldots,x_{h(n)},x_n)$$
where $\mathfrak{S}$ is the Schubert polynomial.
\end{theorem}

We also give an alternate proof that $\Y_{X,h}$ is Cohen-Macaulay and hence equidimensional as a scheme.  This is also implied by recent work of Abe, Fujita, and Zeng ~\cite[Cor. 3.8]{AFZ18}, as we elaborate 
in Remark~\ref{remark: cohen-macaulay}.

We rely on a recent result of Precup stating that, for any regular operator $X$, the dimension of $\Y_{X,h}$ is $\dim(\Y_{X,h})=\sum_{i=1}^n (h(i)-i)$~\cite{Pre16}.  We note that the results of Theorem \ref{theorem: main theorem} also hold for nonregular operators provided the dimension criterion $\dim(\Y_{X,h})=\sum_{i=1}^n (h(i)-i)$ is satisfied.

One feature of our formulas is that the resulting polynomials are manifestly positive sums of monomials.  (For $K$-theory classes, positivity is in the usual sense that the sign of the monomial depends on the parity of the difference between the degree of the monomial and the codimension of the variety.)

The polynomials resulting from our formula generally differ from those obtained from the formulas of Anderson and the second author~\cite{AT10} and Abe, Fujita, and Zeng~\cite{AFZ18}. This is because our formulas hold only in $H^*(G/B)$ or $K^*(G/B)$.  In other words, our formulas agree with previous results only modulo the ideal $I_n=\langle e_d(\xvect)\rangle$ for formulas in cohomology or $J_n=\langle e_d(\xvect)-\binom{n}{d}\rangle$ in $K$-theory.

\begin{example}
Let $n=4$, and let $h$ be the function given by $h(1)=2$, $h(2)=3$, and $h(3)=h(4)=4$.  Then $w_h=12536478\in S_{8}$.  The theorem states that
\begin{align*} [Y_{2344}] &  = \mathfrak{G}_{12536478}(x_1,x_2,x_1,x_3,x_2,x_4,x_3, x_4) \\
= & - x_{1}^{4} x_{2}^{3} x_{3} + x_{1}^{4} x_{2}^{3} + 3 x_{1}^{4} x_{2}^{2} x_{3} + 4 x_{1}^{3} x_{2}^{3} x_{3} - 3 x_{1}^{4} x_{2}^{2} - 4 x_{1}^{3} x_{2}^{3} - 3 x_{1}^{4} x_{2} x_{3} - 12 x_{1}^{3} x_{2}^{2} x_{3} \\ &    - 6 x_{1}^{2} x_{2}^{3} x_{3}  + 3 x_{1}^{4} x_{2} + 12 x_{1}^{3} x_{2}^{2} + 6 x_{1}^{2} x_{2}^{3} + x_{1}^{4} x_{3} + 12 x_{1}^{3} x_{2} x_{3} + 14 x_{1}^{2} x_{2}^{2} x_{3} + 4 x_{1} x_{2}^{3} x_{3} \\ & -  x_{1}^{4} - 10 x_{1}^{3} x_{2} - 13 x_{1}^{2} x_{2}^{2} - 4 x_{1} x_{2}^{3} - 4 x_{1}^{3} x_{3} - 11 x_{1}^{2} x_{2} x_{3} - 6 x_{1} x_{2}^{2} x_{3} -  x_{2}^{3} x_{3} \\&  + 2 x_{1}^{3}  + 7 x_{1}^{2} x_{2} + 4 x_{1} x_{2}^{2} + x_{2}^{3} + 3 x_{1}^{2} x_{3} + 2 x_{1} x_{2} x_{3} + x_{2}^{2} x_{3}  \end{align*} 
 in   $K^*(G/B)$, and 
\begin{align*} 
[Y_{2344}] & = \mathfrak{S}_{12536478}(x_1,x_2,x_1,x_3,x_2,x_4,x_3, x_4)  \\ 
           &= 2x_1^3 + 7x_1^2x_2 + 4x_1x_2^2 + x_2^3 + 3x_1^2x_3  + 2x_1x_2x_3  + x_2^2x_3
\end{align*}
 in $H^*(G/B)$.
\end{example}


\section{Comparing the group to the flag variety}

The central principle in matrix Schubert calculus is that the global algebraic geometry of $G/B$ (for $G=GL_n(\Bbbk)$) and its subvarieties can be studied by looking at $B$-invariant subvarieties of $G$ or, better yet, at $B$-invariant subvarieties of the closure of $G$ (see, e.g., Fulton~\cite{Ful92} or Knutson and Miller~\cite{KM05}).  The closure of $G$ is the space $M_n$ of all $n\times n$ matrices and is isomorphic to $\Bbbk^{n^2}$, so the algebraic geometry of $G/B$ can be studied using commutative algebra in its simplest setting, that of a polynomial ring in $n^2$ variables.

One way to find representatives for classes of subschemes of $G/B$ is to use commutative algebra in $M_n$.  We follow Knutson and Miller's description, starting with standard definitions~\cite[Section 1.2]{KM05},~\cite[Chapter 8]{MS}.  Given a $\mathbb{Z}^\dd$-grading of $S=\Bbbk[z_1,\ldots,z_m]$ and a graded $S$-module $M$, write $\xvect^{\alpha}=\prod_{1 \leq i \leq \dd } x_i^{\alpha_i}$, and let $M_\alpha$ denote the $\alpha$-graded piece of $M$.  The {\em Hilbert series} of $M$ is
$$\mathcal{H}(M;\xvect) = \sum_{\alpha\in\mathbb{Z}^\dd} \dim(M_\alpha) \xvect^{\alpha},$$
and the $K$-polynomial of $M$ is
$$\mathcal{K}(M;\xvect)=\mathcal{H}(M;\xvect)\prod_{i=1}^m (1-\xvect^{\deg(z_i)}).$$

Write $\mathcal{K}(M;\mathbf{1}-\xvect)$ for the $K$-polynomial after substituting $x_i\mapsto (1-x_i)$ for all $i$.  The {\em multidegree} of $M$ (in the $x$-variables) is the polynomial $\mathcal{C}(M;\xvect)$ formed by adding all terms in $\mathcal{K}(M;\mathbf{1}-\xvect)$ of minimal $x$-degree.  (If the module $M$ has dimension $m-r$ over $S$ then this minimal degree equals the ``codimension" $r$ of $M$ \cite[Theorem D]{KM05}.)  

In our context, we use $M_n=\Spec \Bbbk[z_{11},z_{12}, \ldots, z_{1n}, z_{21}, \ldots,z_{nn}]$ with the grading coming from setting $\deg(z_{ij})=e_j$ where $e_j$ is the $j$-th coordinate vector in $\mathbb{Z}^n$.  We let $\pi: G\rightarrow G/B$ denote the natural projection.

Let $Y\subseteq G/B$ be a closed subscheme, and $\overline{Y}\subseteq M_n$ a subscheme such that $\overline{Y}\cap G=\pi^{-1}(Y)$.  Knutson and Miller proved that the multidegree (respectively $K$-polynomial) of the coordinate ring of $\overline{Y}$ represents the class of $Y$ in the cohomology (respectively $K$-theory) of the flag variety of $Y$.  The following is essentially their proof~\cite[Cor. 2.3.1]{KM05}.  We added the last paragraph explicitly identifying each variable $x_j$ in the $K$-polynomial with the class of the $j^{th}$ tautological line bundle $\mathcal{L}_j$ on $G/B$ or in the multidegree with the (first) Chern class $c_1$ of its dual.

\begin{proposition} \label{proposition: relate class to K-polynomial}
Let $Y\subseteq G/B$ be a closed subscheme, and let $\overline{Y}$ be a closed subscheme of $M_n$ such that $\overline{Y}\cap G=\pi^{-1}(Y)$.  Let $\mathcal{K}(\overline{Y})$ be the $K$-polynomial and $\mathcal{C}(\overline{Y})$ the multidegree of $\overline{Y}$ both with the grading coming from setting $\deg(z_{ij})=x_j$.  Then the classes $[Y]$ in the $K$-theory (respectively cohomology) of $G/B$ are represented by $\mathcal{K}(\overline{Y})$ (respectively $\mathcal{C}(\overline{Y})$), where $x_j$ stands for $[\mathcal{L}_j]$ (respectively $c_1([-\mathcal{L}_j])$).
\end{proposition}

\begin{proof} The statement about cohomology follows from the statement about $K$-theory by taking the Chern character.  (The sign change comes from $x_j \mapsto 1-x_j$ in the definition of multidegree.)

The inclusion $G\rightarrow M_n$ induces a surjection $K_T^0(M_n)\rightarrow K_T^0(G)$.  Pulling back vector bundles gives the isomorphism $K^0(G/B)\cong K_B^0(G)$ and restricting the $B$-action to a $T$-action gives the isomorphism $K_B^0(G) \cong K_T^0(G)$.  Call the composition $\phi: K_T^0(M_n)\rightarrow K^0(G/B)$.

The $K$-polynomial $\mathcal{K}(\overline{Y})$ of the coordinate ring of $\overline{Y}$ represents $[\overline{Y}]$ in $K^0_T(M_n)$.  We now prove $\phi([\overline{Y}])=[Y]$.  The surjection $K_T^0(M_n) \rightarrow K_T^0(G)$ sends $[\overline{Y}]$ to $[\overline{Y}\cap G]$.  The group $B$ acts on $\overline{Y}\cap G$. Restricting from the $B$-action to the $T$-action allows us to identify the class of $[\overline{Y}\cap G]$ in $K_T^0(G)$ with that in $K_B^0(G)$.
The flag variety $G/B$ is smooth so there is a resolution $\widetilde{Y}_\bullet$ of $Y$ using vector bundles over $G/B$. The projection $\pi: G \rightarrow G/B$ is flat, so $\pi^{-1}(\widetilde{Y}_\bullet)$ is a resolution of $\pi^{-1}(Y)$ by $B$-equivariant vector bundles. Since $\overline{Y} \cap G = \pi^{-1}(Y)$ by hypothesis, this argument identifies the class $[\overline{Y}\cap G]\in K_B^0(G)$ with $[Y]\in K^0(G/B)$ under the given isomorphism.

Finally we show $\phi(x_j)=[\mathcal{L}_j]$.  In $K_T^0(M_n)$ the variable $x_j$ stands for the degree of the variable $z_{ij}$ or equivalently the class of the principal ideal $[\langle z_{ij}\rangle]$ for any choice of $i$.  The graded module $\langle z_{ij} \rangle$ corresponds to the sheaf of $T$-equivariant sections of the $T$-equivariant line bundle $M_n\times \Bbbk$ on which $T$ acts on the right by $(\mathbf{z},y)\cdot t=(\mathbf{z}\cdot t, yt_{j})$, where $t_j$ is the $j$-th diagonal entry in $T$.  On the other hand $\mathcal{L}_j$ is defined as the quotient of $G\times\Bbbk$ by the equivalence relation $(g, y) \sim (gb, e_j(b)y)$ where $e_j: B\rightarrow\Bbbk^*$ picks out the $j$-th diagonal entry.  This description coincides with the definition of $x_j$ and proves the claim.
\end{proof}

\section{Matrix Schubert varieties and Grothendieck and Schubert polynomials}

This section contains definitions and basic properties of matrix Schubert varieties, including that their Grothendieck and Schubert polynomials are their $K$-polynomials and multidegrees, respectively.  The results in this section come from Knutson and Miller~\cite{KM05}; details and more context can be found there, with notational conventions transposed from ours.  

Let $S=\Bbbk[z_{11},\ldots,z_{nn}]$ and $M_n=\Spec S$.  We think of $M_n$ as the space of all $n\times n$ matrices and $z_{ij}$ as the coordinate function on the $(i,j)$-th entry.  Let $v_j$  denote the $j$-th column of a generic element of $M_n$, and write $v_j=\sum_{i=1}^n z_{ij}e_i$ where $e_i$ is the $i$-th standard basis vector.  (Intuitively $v_j$ can be thought of as the vector-valued function on $M_n$ that picks out the $j$th column of a matrix.)

Fulton defined the {\em Schubert determinantal ideal} $I_w\subseteq S$ and the {\em matrix Schubert variety} $Z_w=\Spec S/I_w\subseteq M_n$ of a permutation $w \in S_n$ as follows~\cite{Ful92}.  For each pair $i,j$ with $1\leq i,j\leq n$, let 
\[r_{ij}(w)=\#(\{w(1),\ldots,w(j)\}\cap\{1,\ldots,i\})\]  
be the rank of the northwest $i\times j$ submatrix of the permutation matrix for $w$  (by which we mean the matrix containing $1$ in entry $(w(i),i)$ and 0's elsewhere). Given a positive integer $m\leq n$, an ordered set $R=(R_1,\ldots,R_m)$ of indices in $\{1,\ldots,n\}$, and an ordered set $C=(C_1,\ldots,C_m)$ of vectors in $\Bbbk^n$, define 
\[d_{R,C} = \det\left(m\times m \textup{ matrix whose} (i,j) \textup{-th entry is the } R_i \textup{-th entry of } C_j\right).\]  
Then let
$$I_{w,i,j}=\langle d_{A,B} \mid A\subseteq\{1,\ldots,i\}, B\subseteq\{v_1,\ldots,v_j\}, \#A=\#B = r_{i,j}(w)+1 \rangle$$
and 
$$I_w=\sum_{i,j=1}^n I_{w,i,j}.$$
In words, $I_w$ is generated by the size $r_{ij}(w)+1$ minors of the northwest $i\times j$ submatrix of a generic element of $M_n$ for all $i$ and $j$.  

For each $w$ the {\em matrix Schubert variety} $Z_w$ is the subvariety of all $n\times n$ matrices whose northwest $i\times j$ submatrices all have rank at most $r_{ij}(w)$. Given a permutation $w$, recall that the length $\ell(w)$ is the minimum number of simple transpositions in a reduced word for $w$, or equivalently $\ell(w)$ is the number of pairs $1 \leq i<j \leq n$ with $w(i)>w(j)$.  Fulton showed the following~\cite{Ful92}.

\begin{proposition} \label{proposition: matrix Schubs}
Fix a permutation $w$ and its matrix Schubert variety $Z_w$.
\begin{enumerate}
\item The matrix Schubert variety is defined (as a reduced scheme) as $Z_w=\Spec S/I_w$ 
\item We have $Z_w \cap G = \pi^{-1}(w_0Y_{w_0w})$ where $w_0Y_{w_0w} := \overline{B^-wB/B}$
\item The matrix Schubert variety $Z_w$ is $\overline{Z_w\cap G}$ and hence is irreducible.
\item The dimension of $Z_w$ is given by $\dim(Z_w)=n^2-\ell(w)$.
\end{enumerate}
\end{proposition}

Not all $i\times j$ submatrices are needed to define the ideal $I_w$.  More precisely, for each $w$
the {\em essential set} $\mathcal{E}(w)$ is
$$\mathcal{E}(w)=\{(i,j)\mid w(j)>i\geq w(j+1), w^{-1}(i)>j\geq w^{-1}(i+1)\}.$$
Fulton also proved the following~\cite{Ful92}. 
\begin{lemma} \label{lemma: essential set decomposition}
The ideal $I_w$ can be written as $I_w=\sum_ {(i,j)\in\mathcal{E}(w)} I_{w,i,j}$.
\end{lemma}

The multidegrees (or equivalently equivariant cohomology classes) of the matrix Schubert varieties are given by {\em Schubert polynomials} as follows.  Denote both the permutation that switches $i$ and $i+1$ and the operator on  $\mathbb{Z}[x_1,\ldots,x_n]$ that switches $x_i$ and $x_{i+1}$ by $s_i$.  Let $\delta_i$ be the operator on $\mathbb{Z}[x_1,\ldots,x_n]$ with
$$\delta_i(f)= \frac{f - s_i f}{x_i - x_{i+1}},$$
 and let $w_0\in S_n$ be the longest permutation, given explicitly by $w_0(i)=n+1-i$ for all $i$.
The {\em Schubert polynomial} $\mathfrak{S}_w$ is defined recursively by $\mathfrak{S}_{ws_i}=\delta_i(\mathfrak{S}_w)$ whenever $ws_i<w$, starting from the highest-degree Schubert polynomial $\mathfrak{S}_{w_0}=\prod_{i=1}^{n-1} x_i^{n-i}$.  The {\em modified Grothendieck polynomial} $\hat{\mathfrak{G}}_w$ is the $K$-polynomial (or equivalently $K$-class) of the matrix Schubert 
variety.  It is defined similarly using the Demazure operator 
$\overline{\delta_i}(f)=-\delta_i(x_{i+1}f)$ and the recursion $\hat{\mathfrak{G}}_{ws_i}=\overline{\delta}_i(\hat{\mathfrak{G}}_w)$ whenever $ws_i<w,$ starting from $\hat{\mathfrak{G}}_{w_0}=\prod_{i=1}^{n-1} (1-x_i)^{n-i}$.\footnote{Knutson and Miller defined the modified Grothendieck polynomials in this way to agree with K-polynomials~\cite{KM05}.  The usual Grothendieck polynomial $\mathfrak{G}_w$ can be recovered from the modified Grothendieck polynomial $\hat{\mathfrak{G}}_w$ by substituting $\mathfrak{G}_w(x_1,\ldots,x_n)=\hat{\mathfrak{G}}_w(1-x_1,\ldots,1-x_n)$.  This substitution commutes with the substitution in our main theorem, so our statements are true for either version of the Grothendieck polynomials (as long as one is consistent).}

Recall that the grading on $S$ in the group $\mathbb{Z}^n$ with generators $x_1,\ldots,x_n$ is defined by $\deg(z_{ij})=x_j$.  We give one of Knutson and Miller's main results~\cite{KM05} (though Feher and Rimanyi proved similar results~\cite{FehRim03}). 

\begin{theorem} \label{theorem: main Schubert theorem}  Fix a permutation $w \in S_n$ and its matrix Schubert variety $Z_w$.
\begin{enumerate}
\item The $K$-polynomial of $Z_w$ is $\hat{\mathfrak{G}}_w$.
\item The multidegree of $Z_w$ is $\mathfrak{S}_w$.
\end{enumerate}
\end{theorem}

By Propositions~\ref{proposition: relate class to K-polynomial}~and~\ref{proposition: matrix Schubs}, it follows that the $K^0$ and cohomology classes of the Schubert variety $Y_{w_0w}$ are represented by $\mathfrak{G}_w$ and $\mathfrak{S}_w$ respectively, recovering seminal results of Bernstein, Gelfand, and Gelfand, Demazure, and Lascoux~and~Sch\"utzenberger \cite{BGG73, Dem74, LS82a,LS82b}.  This new understanding of these classical results was one of the principal motivations of the work of Knutson and Miller.

\section{Regular intersections with Cohen-Macaulay schemes}

The proof of our main theorem is based on the following fact from commutative algebra.  It is clear to experts and seems broadly applicable to many situations, but we could not find it in the literature and so provide a detailed proof.

\begin{theorem} \label{theorem: main step}
Let $S=\Bbbk[z_1,\ldots,z_n]$ be a positively $\mathbb{Z}^d$-graded ring, $\ell_1,\ldots, \ell_a$ linear forms that are homogeneous with respect to the grading, and $R=S/\langle \ell_1,\ldots, \ell_a\rangle$.  Given a graded $S$-module $M$, let $\mathcal{H}_S(M)$ denote the Hilbert series of $M$ with respect to the grading and
$$\mathcal{K}_S(M)=\prod_{i=1}^n (1-\xvect^{\deg(z_i)}) \mathcal{H}_S(M)$$
the $K$-polynomial.  Similarly if $N$ is a graded $R$-module then let $\mathcal{H}_R(N)$ and $\mathcal{K}_R(N)$ denote its Hilbert series and $K$-polynomial with respect to $R$.

Let $M$ be a finitely generated Cohen-Macaulay $S$-module.  Suppose there is a (possibly empty) proper closed subscheme $A\subseteq \Spec R$ such that $\dim (M\otimes_S R)_{\mathfrak{p}}=\dim(M)-a$ for all $\mathfrak{p}\in \Supp (M\otimes_S R)\setminus A$.  Then
$$\mathcal{K}_S(M)=\mathcal{K}_R(M\otimes_S R) + f$$
where $f$ is an element supported on $A$.
%
\end{theorem}

\begin{proof}
Let
$$0\rightarrow F_c \rightarrow F_{c-1} \rightarrow \cdots \rightarrow F_0 \rightarrow M \rightarrow 0$$
be a free resolution of $M$.  Then the $K$-polynomial of $M$ is given by
$$\mathcal{K}_S(M)=\prod_{i=1}^n (1-\xvect^{\deg(z_i)})\mathcal{H}_S(M)=\prod_{i=1}^n (1-\xvect^{\deg(z_i)})\sum_{j=0}^{c} (-1)^j \mathcal{H}_S(F_j)=\sum_{j=0}^c (-1)^j \mathcal{K}_S(F_j).$$
On the other hand, tensoring the above free resolution by $R$ gives
$$\sum_{j=0}^c (-1)^j \mathcal{H}_S(\Tor_j^S (R, M)) = \sum_{j=0}^c (-1)^j \mathcal{H}_S(R\otimes F_j),$$
so
$$\sum_{j=0}^c (-1)^j \mathcal{K}_R(\Tor_j^S (R, M)) = \sum_{j=0}^c (-1)^j \mathcal{K}_R(R\otimes F_j),$$
since for any $R$-module $N$ (which can also be considered as an $S$-module),
$$\mathcal{K}_R(N)=\frac{\prod_{i=1}^n (1-\xvect^{\deg(z_i)})}{\prod_{j=1}^a (1-\xvect^{\deg(\ell_a)})}\mathcal{H}_S(N).$$
As $F_j$ is a finitely generated free $S$-module for all $j$,
$$\mathcal{K}_R(R\otimes F_j)=\mathcal{K}_S(F_j).$$
Hence,
$$\sum_{j=0}^c (-1)^j \mathcal{K}_R(\Tor_j^S (R, M)) = \mathcal{K}_S(M).$$

Suppose $\langle\ell_1,\ldots,\ell_a\rangle\subseteq \mathfrak{p}$ and $\mathfrak{p}\in \Supp(M\otimes R)\setminus A$.  Then since 
\[\dim(M/(\ell_1,\ldots,\ell_a)M)_\mathfrak{p}=\dim(M)_\mathfrak{p}-a,\]
the linear forms $\ell_1,\ldots,\ell_a$ are part of a system of parameters on $M_\mathfrak{p}$.  Moreover  $\ell_1,\ldots,\ell_a$ form a regular sequence on $M_{\mathfrak{p}}$ since $M$ is Cohen--Macaulay~\cite[Theorem 2.1.2]{BH93}. 

Furthermore, since $\ell_1,\ldots,\ell_a$ is a regular sequence on every module of the augmented free resolution of $M_\mathfrak{p}$, tensoring by $R$ is exact on this augmented free resolution~\cite[Theorem 1.1.5]{BH93}.  Thus $\Tor_j^S(R,M)_\mathfrak{p}=0$ for all $j>0$, and so $\Supp \Tor_j^S(R,M)\subseteq A$ for $j>0$.  Therefore,
$$\mathcal{K}_S(M)=\mathcal{K}_R(R\otimes M)+\sum_{j=1}^c (-1)^j \mathcal{K}_R(\Tor_j^S (R, M)) = \mathcal{K}_R(M\otimes R) + f,$$
where $f$ is a signed sum of $K$-polynomials for $R$-modules supported on $A$.
\end{proof}

%

\section{Hessenberg varieties and their equations}

This section defines Hessenberg varieties and gives several results about their equations.  We focus on type $A_n$ but start in a more general setting.  Unfortunately, there are two main definitions in the literature: the first is Lie-theoretic and uses the adjoint action, while the second is more geometric and uses the interpretation of flags as nested subspaces.  Proposition~\ref{prop: Hessenberg description} confirms that these two coincide as sets, which was already known.  The main goal of this section is to  extend Proposition~\ref{prop: Hessenberg description} to a scheme-theoretic result.  This leads us naturally to define the ideals that are the key tools in the proof of our main result (given in the last section of the paper). 

Let $G$ be a semisimple reductive algebraic group over $\Bbbk$ with a fixed Borel subgroup $B$, let $\mathfrak{g}$ be the Lie algebra of $G$, and let $\mathfrak{b}\subset\mathfrak{g}$ be the Lie subalgebra corresponding to $B$.  A {\em Hessenberg space} $H$ is a subspace of $\mathfrak{g}$ containing $\mathfrak{b}$ that is preserved under the adjoint action of $\mathfrak{b}$. In other words $H$ is a subspace with $\mathfrak{b}\subseteq H\subseteq\mathfrak{g}$ and $[b,h]\in H$ for each $b\in\mathfrak{b}$ and $h\in H$.  Given a Hessenberg space $H$ and an element $X\in\mathfrak{g}$, the {\em Hessenberg variety} is the subscheme $\Y_{X,H}\subseteq G/B$ defined by

$$\Y_{X,H}=\{gB\in G/B \mid \Ad(g^{-1})\cdot X\in H\}.$$

There is some confusion in the literature over whether the Hessenberg variety is the subscheme defined by the equations stating that $\Ad(g^{-1})\cdot X\in H$ or the reduced variety supporting this subscheme.  The two definitions do not always coincide~\cite[Theorem 7.6]{AT10}, \cite[Remark 15]{BC04}, though Abe, Fujita, and Zeng prove that they are the same when $X$ is regular and $H$ contains all the negative simple roots (because $\Y_{X,H}$ is reduced in this case)~\cite[Prop. 3.6]{AFZ18}. We use the first definition, which may be more widely accepted, and in any case is the definition that makes our theorems true.  

We now return to $G=GL_n(\Bbbk)$, where there is a slightly different definition.   We will show these two definitions define the same object, even scheme-theoretically.  In this case we can take an element $g\in G$ to be an $n\times n$ invertible matrix.  Let $v_1,\ldots,v_n\in\Bbbk^n$ denote the columns of $g$ from left to right.  An element $g\in GL_n(\Bbbk)$ defines a flag $F_\bullet(g)$ by letting $F_j(g)=\langle v_1,\ldots, v_j\rangle$.
A {\em Hessenberg function} $h:\{1,\ldots,n\}\rightarrow\{1,\ldots,n\}$ is one that satisfies $j\leq h(j)$ for all $j$ with $1\leq j\leq n$ and $h(j)\leq h(j+1)$ for all $j$ with $1\leq j\leq n-1$.  Furthermore, we can consider $X\in\mathfrak{g}$ as a (not necessarily invertible) $n\times n$ matrix.  One can now define
$$\Y^\prime_{X,h}=\{gB\in G/B \mid XF_j(g)\subseteq F_{h(j)}(g) \,\,\forall j\}.$$

Again, one can either define the Hessenberg variety as the subscheme given by the determinantal equations stating these inclusions (or equivalently as a degeneracy locus of a particular map of line bundles) or as the reduced variety supporting this subscheme.

Given a Hessenberg function $h$, the associated Hessenberg space is 
$$H_h=\mathfrak{b}\oplus \bigoplus_{j<i\leq h(j)} \mathfrak{g}_{e_i-e_j}.$$  
Considering $\mathfrak{g}$ as the space of all $n\times n$ matrices, the Hessenberg space becomes
$$H_h=\{x\in\mathfrak{g}\mid x_{ij}=0 \mbox{ for all } i>h(j)\}.$$

The following proposition is well known, but we could not find an explicit proof in the literature.

\begin{proposition} \label{prop: Hessenberg description}
Let $g\in GL_n(\Bbbk)$, $X\in\mathfrak{g}$, and $h$ a Hessenberg function.  Then
$\Y_{X,H_h}=\Y^\prime_{X,h}$ as sets.
\end{proposition}

\begin{proof}
In $GL_n(\Bbbk)$ the adjoint representation is given by conjugation, meaning $gB \in \Y_{X,H_h}$ if and only if $g^{-1}Xg \in H_h$.  The latter is equivalent to the condition that $Xg \in gH_h$ by left multiplication.  Examining this condition column-by-column gives $Xv_j \in \langle v_1, v_2, \ldots, v_{h(j)} \rangle$ for each $j$.  Since $h$ is nondecreasing this is equivalent to $XF_j(g) \subseteq F_{h(j)}(g)$ as desired.
\end{proof}

The previous statements are in the literature.  The main work of this section is to show that $\Y_{X,H_h}=\Y^\prime_{X,h}$ as {\em subschemes} of $G/B$ by showing that the ideals defining $\pi^{-1}(\Y_{X,H_h})$ and $\pi^{-1}(\Y^\prime_{X,h})$ coincide.  The statement then follows by the correspondence between $B$-invariant subschemes of $G$ and subschemes of $G/B$ (which can be shown, for example, by considering the universal property of quotients).  We will first construct the ideal defining $\pi^{-1}(\Y_{X,H_h})$, then the ideal defining $\pi^{-1}(\Y^\prime_{X,h})$, and finally show that the ideals coincide in the coordinate ring of $G$ using a determinantal relation.

Let $R=\Bbbk[z_{ij}]$ for $1\leq i,j\leq n$ be the coordinate ring for the space $M_n$ of $n\times n$ matrices, so that the $j$-th column is $v_j=\sum_{i=1}^n z_{ij}e_i$ where $e_i$ is the $i$-th standard basis vector in $\Bbbk^n$.  Then $G=\Spec(R[d^{-1}])$, where $d=d_{(1,\ldots,n),(v_1,\ldots,v_n)}$ is the determinant of the generic matrix.
We now define an ideal $I_{X,H_h}\subseteq R$ generated by the explicit equations for the condition that $\Ad(g^{-1})\cdot X=g^{-1}Xg\in H_h$.  Hence we will have $\pi^{-1}(\Y_{X,H_h})=\Spec(R/I_{X,H_h})\cap G$ by definition.

By Cramer's Rule, the $(i,k)$-th entry of $g^{-1}$ is  
\[(-1)^{i+k}d_{(1,\ldots,k-1,k+1,\ldots,n),(v_1,\ldots,v_{i-1},v_{i+1},\ldots,v_n)}/d.\]  
To obtain the $(i,j)$-th entry of $g^{-1}Xg$, we insert $Xv_j$ into this determinant by Laplace expansion, resulting in $d_{(1,\ldots,n),(v_1,\ldots,v_{i-1},Xv_j,v_{i+1},\ldots,v_n)}/d$.  If we define
$$I_{X,H_h}=\langle d_{(1,\ldots,n),(v_1,\ldots,v_{i-1},Xv_j,v_{i+1},\ldots,v_n)}\mid i>h(j)\rangle,$$
then $\pi^{-1}(\Y_{X,h})=\Spec(R/I_{X,H_h})\cap G$.

To define an ideal cutting out $\pi^{-1}(\Y^\prime_{X,h})$, note that the condition $XF_j(g)\subseteq F_{h(j)}(g)$ is equivalent to requiring that the rank of $\{Xv_1,\ldots,Xv_j,v_1,\ldots,v_{h(j)}\}$ be exactly $h(j)$, which in turn can be expressed in equations by insisting that the size $h(j)+1$ minors of the matrix with $\{Xv_1,\ldots,Xv_j,v_1,\ldots,v_{h(j)}\}$ as the columns all vanish.  Hence, for any $j$ and $r$, let
$$J_{X,j,r}=\langle d_{R,C} \mid R\subseteq\{1,\ldots,n\}, C\subseteq \{Xv_1,\ldots,Xv_j,v_1,\ldots,v_r\}, \#R=\#C=r+1\rangle,$$ and define
$$J_{X,h}=\sum_{j=1}^n J_{X,j,h(j)}.$$
Then we have $\pi^{-1}(\Y^\prime_{X,h})=\Spec(R/J_{X,h})\cap G$, again by definition.

We are almost ready to show that $I_{X,H_h}$ and $J_{X,h}$ agree when considered as ideals in the coordinate ring of $G$.  The last tool we need is the following, known as the Pl\"ucker relation.  We provide a proof since we need it in a slightly more general form than usual; our proof is essentially that of~\cite[Lemma 8.1.2]{Ful97}.

\begin{lemma}  \label{lemma: Plucker}

Let $u_1,\ldots,u_m,v_1,\ldots, v_n\in \Bbbk^n$ and let $\phi: \Bbbk^n\rightarrow\Bbbk^m$ be a linear map.
Let $M$ be the $m\times m$ matrix with column vectors $\phi(u_1),\ldots,\phi(u_m)$ and $N$ the $n\times n$ matrix with column vectors $v_1,\ldots,v_n$.  Fix $k$ with $1\leq k\leq m$, and let $M_i$ be the matrix obtained from $M$ by replacing $\phi(u_k)$ with $\phi(v_i)$, while $N_i$ is the matrix obtained from $N$ by replacing $v_i$ with $u_k$.  Then
$$\det(M)\det(N)=\sum_{i=1}^n \det(M_i)\det(N_i).$$
\end{lemma}

\begin{proof}
Consider the $(n+1)\times (n+1)$ matrix $\mathcal{N}$ formed by taking $N$, adding $u_k$ as the first column, and adding a top row whose entries are $\det(M),\det(M_1),\ldots,\det(M_n)$.  By Laplace expansion along the top row
$$\det(\mathcal{N})=\det(M)\det(N)-\sum_{i=1}^n \det(M_i)\det(N_i).$$
(There are no signs in the sum since the sign in the Laplace expansion cancels with the sign needed to rearrange the columns of a minor of $\mathcal{N}$ into their order in $N_i$.)

On the other hand, we can show that $\mathcal{N}$ is singular.  Laplace expansion along the $k$-th column gives
$$\det(M_i)=\sum_{j=1}^m (-1)^j \phi_j(v_i)\det(M^j)$$
and
$$\det(M)=\sum_{j=1}^m (-1)^j \phi_j(u_k)\det(M^j),$$
where $\phi_j$ is $\phi$ composed with projection to the $j$-th entry and $M^j$ is the submatrix of $M$ formed by deleting the $j$-th row and the $k$-th column.  Hence the first row of $\mathcal{N}$ is a linear combination of the other rows since each $\phi_j$ is linear.  Therefore
$$\det(\mathcal{N})=\det(M)\det(N)-\sum_{i=1}^n \det(M_i)\det(N_i)=0,$$
and
$$\det(M)\det(N)=\sum_{i=1}^n \det(M_i)\det(N_i).  $$  
\end{proof}

We prove the next theorem by determinantal calculations.

\begin{theorem} \label{theorem: two Hessenberg ideals}
For any Hessenberg function $h$
$$I_{X,H_h}=J_{X,h}$$
as ideals in $\Bbbk[z_{ij}, d^{-1}]$, namely the coordinate ring of $GL_n(\Bbbk)$.
\end{theorem}

\begin{proof}
We first show that $I_{X,H_h} \subseteq J_{X,h}$ by proving
$$d_{(1,\ldots,n),(v_1,\ldots,v_{i-1},Xv_j,v_{i+1},\ldots,v_n)}\in J_{X,h}$$
for all $i>h(j)$.  We use the generalized Laplace expansion $d_{(1,\ldots,n),(v_1,\ldots,v_{i-1},Xv_j,v_{i+1},\ldots,v_n)}$ simultaneously along the columns $v_1,\ldots,v_{h(j)}, Xv_j$, which gives
$$d_{(1,\ldots,n),(v_1,\ldots,v_{i-1},Xv_j,v_{i+1},\ldots,v_n)}=\sum_R (-1)^{\sum_{k\in R} k -\binom{h(j)+2}{2}} d_{R,(v_1,\ldots,v_{h(j)}, Xv_j)} d_{(1,\ldots,n)\setminus R, (v_{h(j)+1},\ldots, v_{i-1}, v_{i+1},\ldots, v_n)},$$
where the sum is over all subsets $R\subseteq\{1,\ldots, n\}$ of size $h(j)+1$.  An expansion like this only requires $i \neq h(j)$ but the form we wrote assumes additionally that $i>h(j)$.  By the definition of $J_{X,j,h(j)}$ we know $d_{R,(v_1,\ldots,v_{h(j)}, Xv_j)}\in J_{X,j,h(j)}$ for all $R$.  By definition $J_{X,j,h(j)}\subseteq J_{X,h}$ so 
$$d_{(1,\ldots,n),(v_1,\ldots,v_{i-1},Xv_j,v_{i+1},\ldots,v_n)}\in J_{X,h}.$$
(This part of our proof works over $\Bbbk[z_{ij}]$, the coordinate ring of the space of $n\times n$ matrices.)

Now we show that $J_{X,h} \subseteq I_{X,H_h}$.  It suffices to prove $d_{R,C}\in I_{X,H_h}$ whenever $R\subseteq\{1,\ldots,n\}$, $C\subseteq \{Xv_1,\ldots,Xv_j, v_1,\ldots, v_{h(j)}\}$, and $\#R=\#C=h(j)+1$.  
Let $d=d_{(1,\ldots,n),(v_1,\ldots,v_n)}$, and let $\ell$ be the maximum index such that $Xv_{\ell}\in C$. For each $i$ let $C_i$ be the multiset defined by $C_i=( C\setminus\{Xv_{\ell}\}) \cup\{v_i\}.$
Then
$$d_{R,C}d=\sum_{i=1}^n d_{R, C_i} d_{(1,\ldots,n), (v_1,\ldots, v_{i-1}, Xv_{\ell}, v_{i+1},\ldots, x_n)}$$
 by the Pl\"ucker relation in Lemma~\ref{lemma: Plucker}.   
(In Lemma~\ref{lemma: Plucker}, the matrix $d_{R,C}$ is the $m \times m$ matrix constructed from certain rows of the subset $C$ of columns, while the matrix $d$ is given by the $n$ columns $v_1, v_2, \ldots, v_n$ and the map $\phi$ selects the entries in the rows identified by $R$.  In other words, each $v_i$ in this theorem coincides with $v_i$ from Lemma~\ref{lemma: Plucker} while the vectors $u_1,\ldots,u_m$ are the elements of $C$.)
 
We now induct on $k$ where $k=\#\left(\{Xv_1,\ldots,Xv_j\}\cap C\right)$.  In the base case when $k=1$, either $i>h(\ell)$ or $i\leq h(\ell)$.  In the first case $d_{(1,\ldots,n), (v_1,\ldots, v_{i-1}, Xv_{\ell}, v_{i+1},\ldots, x_n)}\in I_{X,H_h}$ by definition, and in the second $d_{R,C_i}=0$ since $C_i$ contains $v_i$ twice.  Hence every term on the right hand side is in $I_{X,H_h}$, so $d_{R,C}d\in I_{X,H_h}$.  Thus $d_{R,C}\in I_{X,H_h}$ because $d$ is a unit.

The inductive hypothesis states that $d_{R',C'}\in I_{X,H_h}$ whenever $\#\left(\{Xv_1,\ldots,Xv_j\}\cap C'\right)=k-1$.  Now consider $d_{R,C}$.  If $i\leq h(\ell)$ then $C_i$ contains $k-1$ vectors of the form $Xv_m$ for $m\leq j$. (It also contains $h(j)-k+2$ vectors of the form $v_m$ for $m\leq h(j)$.)  Thus by induction $d_{R,C_i}\in I_{X,H_h}$ in this case.  If $i>h(\ell)$ then again $d_{(1,\ldots,n), (v_1,\ldots, v_{i-1}, Xv_{\ell}, v_{i+1},\ldots, x_n)}\in I_{X,H_h}$ by definition.  Hence every term in the right hand side is in $I_{X,H_h}$, and so $d_{R,C}\in I_{X,H_h}$.
\end{proof}

\section{Matrix Hessenberg varieties and the main theorem}

We now prove Theorem \ref{theorem: main theorem} by realizing $J_{X,h}$ as the tensor product of a Schubert determinantal ideal with a quotient by a linearly generated ideal and  applying Theorem \ref{theorem: main step}.

Let $M_{2n}$ be the space of $2n\times 2n$ matrices, and denote the coordinate ring $S=\Bbbk[y_{ij}]$ for $1\leq i,j\leq 2n$.  Let $R=\Bbbk[z_{ij}]$ for $1\leq i,j\leq n$ be the coordinate ring of $M_n$.  Given a linear operator $X$ and a Hessenberg function $h:\{1,\ldots,n\}\rightarrow\{1,\ldots,n\}$ we now define a ring homomorphism $\phi_{X,h}: S\rightarrow R$.  
Intuitively, $\phi_{X,h}$ will be induced from the map $M_n\rightarrow M_{2n}$ given by starting with the ordered columns $v_1, v_2, \ldots, v_n$, and then for each $j$ inserting $Xv_j$ immediately after the column with $v_{h(j)}$ (and after all $Xv_1, \ldots, X v_{j-1}$ if there are any other $j'$ for which $h(j')=h(j)$).  The bottom half of the matrix in $M_{2n}$ is then filled with zeroes, and the  rank conditions imposed on this matrix are equivalent to the geometric definition of the Hessenberg variety.

More precisely, we do as follows.  If $i>n$ then $\phi_{X,h}(y_{ij})=0$.  Otherwise:
\begin{itemize}
\item If $j=m+h(m)$ for some $m$, then $\phi_{X,h}(y_{ij})$ is the $i$-th entry of $Xv_j$ where $v_j=\sum_{i=1}^n z_{ij} e_i$ is the $j$-th column of a generic element of $M_n$.
\item If there is no $m$ for which $j= m+h(m)$, then $j=m+h^\prime(m)$ for some $m$, where $h^\prime(m)=\#\{p\mid h(p)<m\}$.  In this case $\phi_{X,h}(y_{ij})=z_{im}$.
\end{itemize}
Note that for any $X$ and $h$, the map $\phi_{X,h}$ is surjective, and $\ker \phi_{X,h}$ is generated by $3n^2$ independent linear forms. 

Recall from the introduction that we define a permutation $w_{h}\in S_{2n}$ from a Hessenberg function $h$ by letting $w_h(m+h(m))=n+m$ for each $m$ and putting $1,\ldots,n$ in the other entries in order.  

We need the following lemma about $w_h$.

\begin{lemma} \label{lemma: Hess perm length}
The length of $w_h$ is 
$$\ell(w_h)=\sum_{i=1}^n (n-h(i)).$$
\end{lemma}

\begin{proof}
Since $w_h^{-1}(n+1)<\cdots<w_h^{-1}(2n)$ and $w_h^{-1}(1)<\cdots<w_h^{-1}(n)$ by definition, the only possible inversions in $w_h$ are at indices $i<j$ where $w_h(i)>n$ and $w_h(j)\leq n$.  If $i=m+h(m)$ then there are $2n-m-h(m)$ indices with $j>i$, of which $n-m$ have $w_h(i)<w_h(j)$.  Hence $n-h(m)$ of them are inversions, and $w_h$ has $\sum_{i=1}^n (n-h(i))$ inversions.
\end{proof}

Finally we show that the $J_{X,h}$ can be written as follows.

\begin{proposition} \label{proposition: main step}
For any linear operator $X$ and Hessenberg function $h$ $$J_{X,h}=\phi_{X,h}(I_{w_h}).$$
\end{proposition}

\begin{proof}
Recall that the essential set of a permutation $w$ is given by
$$\mathcal{E}(w)=\{(i,j)\mid w(j)>i\geq w(j+1), w^{-1}(i)>j\geq w^{-1}(i+1)\}.$$
If $h(m)=n$ for all $m$ then $E(w_h)=\emptyset$ and $I_{w_h}=J_{X,h}=\langle 0\rangle$.
Otherwise,
\begin{itemize}
\item $w_h^{-1}(i)>w_h^{-1}(i+1)$ if and only if $i=n$, and 
\item $w_h(j)>w_h(j+1)$ if and only if $j=m+h(m)$ for some $m$ such that $h(m)<h(m+1)$, in which case
$w_h(j)>n\geq w_h(j+1)$ and $w_h^{-1}(n)>j\geq w_h^{-1}(n+1)$. 
\end{itemize}
Hence by Lemma \ref{lemma: essential set decomposition}
$$I_{w_h}=\sum_{m \textup{ s.t. }h(m)<h(m+1)} I_{w_h,n,m+h(m)}.$$
Furthermore $r_{n,m+h(m)}(w_h)=h(m)$ so  
$$I_{w_h,n,m+h(m)}=\langle d_{R,C} \mid R\subseteq\{1,\ldots,n\}, C\subseteq\{u_1,\ldots,u_{m+h(m)}\}, \#R=\#C = h(m)+1\rangle,$$ 
where $u_j=\sum_{i=1}^n y_{ij} e_i$ is the vector given by the top half of the $j$-th column of a generic element of $M_{2n}$. 

The image $\phi_{X,h}(I_{w_h,n,m+h(m)})$  can be written explicitly as
$$\phi_{X,h}(I_{w_h,n,m+h(m)})=\langle d_{R,C} \mid R\subseteq\{1,\ldots,n\}, C\subseteq\{v_1,\ldots,v_{h(m)}, Xv_1,\ldots,Xv_m\}, \#R=\#C= h(m)+1\rangle,$$
which is $J_{X,m,h(m)}$ by definition.  When $h(m)=h(m+1)$ we have $J_{X,m,h(m)}\subseteq J_{X,m+1,h(m+1)}$ also by definition, and so
$$J_{X,h}=\sum_{m \textup{ s.t. } h(m)<h(m+1)} J_{X,m,h(m)}=\phi_{X,h}(I_{w_h}).$$
\end{proof}

We now prove Theorem \ref{theorem: main theorem} by assembling the key steps above. 

\begin{proof}
Give $S$ the grading such that $\deg(y_{ij})=x_m$ if $j=m+h(m)$ or $j=m+h^\prime(m)$, and give $R$ the grading such that $\deg(z_{ij})=x_j$.  This grading makes $\phi_{X,h}$ a graded ring homomorphism.  By Theorem \ref{theorem: main Schubert theorem}, we have
$$[S/I_{w_h}]=\mathfrak{G}_{w_h}(x_1,\ldots,x_{h(1)},x_1,x_{h(1)+1},\ldots,x_{h(2)},x_2,x_{h(2)+1},\ldots,x_{h(n)},x_n)$$
in $K_T^0(M_{2n})$.
By Proposition~\ref{proposition: main step} we know $R/J_{X,h} = \phi_{X,h}(S/I_{w_h})$.

Since $X$ is regular we can use Precup's dimension result~\cite[Cor. 2.7]{Pre16} to conclude $\dim(\Y_{X,h})=\sum_{i=1}^n (h(i)-i)$.
Hence, since, by Theorem~\ref{theorem: two Hessenberg ideals} $J_{X,h}$ is the defining ideal for $\pi^{-1}(\Y_{X,h})$, we conclude
$$\dim((R/J_{X,h})_\mathfrak{p})=\dim(B)+\dim(\Y_{X,h})=\binom{n+1}{2}+\sum_{i=1}^n (h(i)-i)=\sum_{i=1}^n h(i)$$
for any $\mathfrak{p}\in G$.
On the other hand, by Proposition~\ref{proposition: matrix Schubs} and Lemma~\ref{lemma: Hess perm length},
$$\dim(S/I_{w_h})=4n^2-\sum_{i=1}^n (n-h(i)) = 3n^2+\sum_{i=1}^n h(i).$$
It follows that $\dim((R/J_{X,h})_\mathfrak{p}) = \dim(S/I_{w_h}) - 3n^2$ for any $\mathfrak{p}\in G$.

Since $\ker(\phi_{X,h})$ is generated by $3n^2$ independent linear forms, we can apply Theorem \ref{theorem: main step} with $M=S/I_{w_h}$, $a=3n^2$, and $A=M_n\setminus G$ to get that
$$[R/J_{X,h}]=\mathfrak{G}_{w_h}(x_1,\ldots,x_{h(1)},x_1,x_{h(1)+1},\ldots,x_{h(2)},x_2,x_{h(2)+1},\ldots,x_{h(n)},x_n)+f$$
in $K_T^0(M_n)$ for some $f\in K_T^0(M_n\setminus G)$.  By Theorem \ref{theorem: two Hessenberg ideals} and Proposition \ref{proposition: relate class to K-polynomial} we see that 
\[\mathfrak{G}_{w_h}(x_1,\ldots,x_{h(1)},x_1,x_{h(1)+1},\ldots,x_{h(2)},x_2,x_{h(2)+1},\ldots,x_{h(n)},x_n)\]
represents the class of the Hessenberg variety $\Y_{X,h}$ modulo the kernel of $K_T^0(M_n)\rightarrow K_T^0(G)$.

The statement for cohomology follows by taking the Chern character map.
\end{proof}

\begin{remark} \label{remark: cohen-macaulay}
Recent work of Abe, Fujita, and Zeng \cite{AFZ18} implies that regular Hessenberg varieties, by our definition, are always Cohen-Macaulay.  Indeed, the definition of $\Y_{X,H}$ shows that the scheme is a local complete intersection, and hence Cohen-Macaulay, whenever it meets the dimension constraints; Precup's result shows this dimension constraint  holds for regular operators \cite{Pre16}.  Abe, Fujita, and Zeng's result is that the scheme is reduced in the case when $h(i)>i$ for all $i$, and thus the reduced variety is Cohen-Macaulay in this case.

We now sketch an alternate proof of this result using the results in this paper.

Fulton proved that each matrix Schubert variety is Cohen-Macaulay~\cite{Ful92}.  Thus $S/I_{w_h}$ is Cohen-Macaulay. Recall that Theorem \ref{theorem: main step} used the fact that $M$ is Cohen--Macaulay to find $\ell_1,\ldots,\ell_a$ that form a regular sequence on $M_{\mathfrak{p}}$~\cite[Theorem 2.1.2]{BH93}. It follows that $(M\otimes_S R)_\mathfrak{p}$ is Cohen--Macaulay for any $\mathfrak{p} \not \in A$~\cite[Theorem 2.1.3]{BH93}. Using the same $M$ and $a$ as in Theorem~\ref{theorem: main theorem}, we conclude $R/J_{X,h}$ is Cohen--Macaulay at $\mathfrak{p}$ for any $\mathfrak{p}\in G$.  Finally, since $\pi: G \rightarrow G/B$ is a fiber bundle with fibers isomorphic to $B$, given any closed subscheme $Y \subseteq G/B$, the local ring
\[\mathcal{O}_{\pi^{-1}(Y),g} \cong \mathcal{O}_{Y,gB} \otimes \mathcal{O}_{B,e}\]
 for any $g$.
Since $B$ is smooth, $Y$ is Cohen-Macaulay if and only if $\pi^{-1}(Y)$ is.  In particular $\Y_{X,h}$ is Cohen-Macaulay.
\end{remark}

\section*{Acknowledgements}

We thank Martha Precup for suggesting that our work applies to all regular Hessenberg varieties.  AW was a sabbatical visitor at the University of Illinois at Urbana--Champaign and the University of Washington while this work was completed, and he thanks their Departments of Mathematics for their hospitality.  We also thank the anonymous referees for their helpful remarks.

\end{document}